\documentclass[12pt,leqno,fleqn]{amsart}
\usepackage{amsmath,amstext,amsthm,amssymb,amsxtra}

\usepackage{txfonts} 
\usepackage[T1]{fontenc}
\usepackage{lmodern}

\usepackage{euler}
\usepackage{graphicx}
\usepackage{transparent}
\usepackage[normalem]{ulem}

\usepackage[bookmarksopen,bookmarksdepth=3,colorlinks,citecolor=red,pagebackref,hypertexnames=true]{hyperref}
\usepackage[backrefs,msc-links,nobysame,initials,non-sorted-cites]{amsrefs}

\usepackage{url}

\DefineSimpleKey{bib}{myurl}
\newcommand\myurl[1]{\url{#1}}

\BibSpec{webpage}{%
  +{}{\PrintAuthors} {author}
  +{,}{ \textit} {title}
  +{}{ \parenthesize} {date}
  +{,}{ \myurl} {myurl}
  +{,}{ } {note}
  +{.}{ } {transition}
}

\allowdisplaybreaks

\theoremstyle{plain}

\newtheorem{lemma}[equation]{Lemma}
\newtheorem{proposition}[equation]{Proposition}

\newtheorem{theorem}[equation]{Theorem}
\newtheorem*{theorem*}{Theorem}
\newtheorem{corollary}[equation]{Corollary}
\theoremstyle{definition}

\numberwithin{equation}{section}

\def\norm#1.#2.{\lVert#1\rVert_{#2}}
\def\Norm#1.#2.{\bigl\lVert#1\bigr\rVert_{#2}}
\def\NOrm#1.#2.{\Bigl\lVert#1\Bigr\rVert_{#2}}
\def\NORm#1.#2.{\biggl\lVert#1\biggr\rVert_{#2}}
\def\NORM#1.#2.{\Biggl\lVert#1\Biggr\rVert_{#2}}

\def\ind{\textnormal{\textbf 1}}

\def\R{\mathbb R}

\def\M{\mathsf{M}}
\def\Cw{\mathsf{C_w ^+}}

\def\ip#1,#2,{\langle #1,#2\rangle}
\def\Ip#1,#2,{\bigl\langle#1,#2\bigr\rangle}
\def\IP#1,#2,{\Bigl\langle#1,#2\Bigr\rangle}

\def\Abs#1{\bigl\lvert#1\bigr\rvert}

\begin{document}
\raggedbottom
\title{Sharp inequalities for one-sided Muckenhoupt weights}

\begin{abstract} Let $A_\infty ^+$ denote the class of one-sided Muckenhoupt weights, namely all the weights $w$ for which $\M^+:L^p(w)\to L^{p,\infty}(w)$ for some $p>1$, where $\M^+$ is the forward Hardy-Littlewood maximal operator. We show that $w\in A_\infty ^+$ if and only if there exist numerical constants $\gamma\in(0,1)$ and $c>0$ such that
\[	w(\{x \in \mathbb{R} : \,  \M^+\ind_E (x)>\gamma\})\leq cw(E)\]
for all measurable sets $E\subset \R$. Furthermore, letting
\[
\Cw(\alpha)\coloneqq \sup_{0<w(E)<+\infty} \frac{1}{w(E)} w(\{x\in\R:\,\M^+\ind_E(x)>\alpha\})
\]
we show that for all $w\in A_\infty ^+$ we have the asymptotic estimate $\Cw(\alpha)-1\lesssim (1-\alpha)^\frac{1}{c[w]_{A_\infty ^+}}$ for $\alpha$ sufficiently close to $1$ and $c>0$ a numerical constant, and that this estimate is best possible. We also show that the reverse H\"older inequality for one-sided Muckenhoupt weights, previously proved by Mart\'in-Reyes and de la Torre, is sharp, thus providing a quantitative equivalent definition of $A_\infty ^+$. Our methods also allow us to show that a weight $w\in A_\infty ^+$ satisfies $w\in A_p ^+$ for all $p>e^{c[w]_{A_\infty ^+}}$.
\end{abstract}

\author{Paul Hagelstein}
\address{P.H.: Department of Mathematics, Baylor University, Waco, Texas 76798}
\email{\href{mailto:paul_hagelstein@baylor.edu}{paul\!\hspace{.018in}\_\,hagelstein@baylor.edu}}
\thanks{P. H. is partially supported by a grant from the Simons Foundation (\#208831 to Paul Hagelstein).}

\author{Ioannis Parissis}
\address{I.P.: Departamento de Matem\'aticas, Universidad del Pais Vasco, Aptdo. 644, 48080
Bilbao, Spain and Ikerbasque, Basque Foundation for Science, Bilbao, Spain}
\email{\href{mailto:ioannis.parissis@ehu.es}{ioannis.parissis@ehu.es}}
\thanks{I. P. is supported by grant  MTM2014-53850 of the Ministerio de Econom\'ia y Competitividad (Spain), grant IT-641-13 of the Basque Gouvernment, and IKERBASQUE}

\author{Olli Saari}
\address{O.S.: Department of Mathematics and Systems Analysis, Aalto University, FI-00076 Aalto, Finland
} \email{\href{mailto:olli.saari@aalto.fi}{olli.saari@aalto.fi}}
\thanks{O. S. is supported by the Academy of Finland and the V\"ais\"al\"a Foundation.}

\maketitle

\section{Introduction} \label{s.introduction} We are interested in topics related to one-sided maximal operators on Euclidean spaces. Our focus is on the one-dimensional case and the main operator under study in this paper is the forward one-sided Hardy-Littlewood maximal operator defined by
\[
\M^+f(x)\coloneqq \sup_{h>0}\frac {1}{h} \int_x ^{x+h} |f(t)|dt,\quad x\in\R,\quad f\in L^1 _{\text{loc}}(\R).
\]
By a weight $w$ we always mean a non-negative, locally integrable function on the real line. The weights $w$ for which $\M^+:L^p(w)\to L^{p,\infty}(w)$ have been identified and extensively studied. See for example \cites{Reyes, MPT, MOT,MT, ORT, Saw}. Thus it is well known that the appropriate one-sided Muckenhoupt class can be defined by
\[
[w]_{A_p ^+}\coloneqq \sup_{a<b<c}\frac{w(a,b)}{|(a,c)|} \Big( \frac{\sigma(b,c)}{|(a,c)|}\Big) ^{p-1},
\]
where $\sigma\coloneqq w^{-\frac{1}{p-1}}$, and we have that $\M^+:L^p(w)\to L^{p,\infty}(w)$, $1<p<\infty$, if and only if $w\in A_p ^+$. For $p=1$ we also have the endpoint result that $\M^+:L^1(w)\to L^{1,\infty}(w)$ if and only if
\[
[w]_{A_1 ^+}\coloneqq \NOrm \frac{\M^-w}{w}. L^\infty(\R).<+\infty,
\]
where $\M^-f(x)\coloneqq \sup_{h>0} \frac {1}{h}\int_{x-h} ^x |f(t)|dt$; see for example \cites{MOT,Saw}.

The corresponding one-sided $A_\infty ^+$ class has also been studied and one can define for example
\[
A_\infty ^+ \coloneqq \bigcup_{p>1} A_p ^+;
\]
see \cite{MPT}. The class $A_\infty ^+$ can be defined in many other equivalent ways. Here we adopt the definition through a Fujii-Wilson-type constant which amounts to demanding that
\[
[w]_{A_\infty ^+}\coloneqq \sup_{a<b}   \frac{1}{w(a,b)}\int_{(a,b)} \M^-(w\ind_{(a,b)})<+\infty.
\]
We note here that there is no standard definition for the $A_\infty ^+$ constant of one-sided Muckenhoupt weights $w$. The definition above appears in \cite{MT}*{Definition 1.7} where it is shown that $w\in A_\infty ^+$ if and only if $[w]_{A_\infty ^+}<+\infty$. A similar definition appears in \cite{KS}. We note also that for two-sided weights the constant above was introduced independently by Fujii \cite{Fujii} and Wilson \cites{W1,W2}. In the two-sided case the constant above was shown to be the appropriate quantity in order to prove sharp weighted norm inequalities for singular integral and maximal operators; see for example \cite{HytP}. Similar results for the one-sided maximal operator appear in \cite{MT}.

In this paper we pursue a characterization of $A_\infty ^+$ in terms of \emph{tauberian constants}, in the spirit of \cites{HLP,HP2}. Given a locally integrable function $w\in A_\infty ^+$ and $\alpha\in(0,1)$ we define
\[
\Cw(\alpha)\coloneqq \sup_{0<w(E)<+\infty} \frac{1}{w(E)} w(\{x\in\R:\,\M^+\ind_E(x)>\alpha\}).
\]
We call $\Cw(\alpha)$ the \emph{sharp weighted tauberian constant} of $\M^+$. It is obvious that $\Cw(\alpha)<+\infty$ for all $\alpha\in(0,1)$ whenever $\M^+:L^p(w)\to L^{p,\infty}(w)$. In this paper we show that an apparently much weaker converse to this statement holds, namely that $\Cw(\alpha_o)<+\infty$ for a single value $\alpha_o\in(0,1)$ already implies that $M^+:L^p(w)\to L^p(w)$ for sufficiently large $p>1$ and thus that $w\in A_\infty ^+$.
This is the content of our main theorem.

\begin{theorem}\label{t.main} Let $w$ be a non-negative, locally integrable function on the real line. The following are equivalent:
	\begin{enumerate}
	\item[(i)] We have that $w\in A_\infty ^+$.
	\item[(ii)] There exists $\delta>0$ such that $\Cw(\alpha)-1\lesssim (1-\alpha)^\delta$ as $\alpha\to 1^-$.
	\item[(iii)] There exists $\alpha_o\in (0,1)$ such that $\Cw(\alpha_o)<+\infty$.	
	\end{enumerate}
	
\end{theorem}

The study of asymptotic estimates of the type $\Cw(\alpha)-1\lesssim (1-\alpha)^\delta$ as $\alpha\to 1^-$ has a small history. In \cite{Sol}, Solyanik proved corresponding estimates for the usual Hardy-Littlewood maximal function with respect to axes parallel cubes in $\R^n$, both in its centered and non-centered version, as well as for the strong maximal function. In \cites{HP1} the first two authors continued these investigations and introduced the term \emph{Solyanik estimate} to indicate the validity of $\lim_{\alpha\to 1^-}\mathsf{C}(\alpha)=1$, whenever $\mathsf{C}(\alpha)$ denotes the sharp tauberian constant with respect to some geometric maximal operator. Finally, in \cite{HP2}, Solyanik estimates are shown to hold in the classical (two-sided) weighted setting and in fact they characterize the usual $A_\infty$ class of Muckenhoupt weights.

 Going back to our main result above and assuming that $w\in A_\infty ^+$ we can give a very precise version of the Solyanik estimate (ii).
\begin{corollary}\label{c.main1} Let $w\in A_\infty ^+$ be a one-sided Muckenhoupt weight on the real line. There exists a numerical constant $c>0$ such that
	\[
	\Cw(\alpha)-1\lesssim  (1-\alpha)^{(c[w]_{A_\infty ^+})^{-1}}\quad\text{for all}\quad 1>\alpha>1-e^{-c[w]_{A_\infty ^+}}
	\]
with the implicit constant independent of everything. Furthermore this estimate is optimal up to numerical constants; if $w$ is a locally integrable function that satisfies
 \[
 \Cw(\alpha)-1\lesssim (1-\alpha)^\frac{1}{\beta}\quad\text{whenever}\quad 1>\alpha>1-e^{-\beta}
 \]
 for some constant $\beta>1$ then $w\in A_\infty ^+$ and $[w]_{A_\infty ^+}\lesssim \beta$.
\end{corollary}
We note in passing that the sharpness direction in Corollary~\ref{c.main1} above relies on exhibiting a sharp reverse H\"older inequality for one-sided Muckenhoupt weights. The validity of the reverse H\"older inequality is actually proved in \cite{MT}. We show in Theorem~\ref{t.ainfty} that the reverse H\"older inequality of \cite{MT} is best possible, which in turn allows to prove the optimality of the Solyanik estimate in the corollary above.

Secondly, it is of some importance to note that our method of proof together with the corollary above allows us to conclude a quantitative embedding of $A_\infty ^+$ into $A_p ^+$.
\begin{corollary}\label{c.main2} Let $w\in A_\infty ^+$ be a one-sided Muckenhoupt weight on the real line and denote by $[w]_{A_\infty ^+}$ its Fujii-Wilson constant. Then there exists a numerical constant $c>0$ such that $w\in A_p ^+$ for all $p>e^{c[w]_{A_\infty ^+}}$ and $[w]_{A_p ^+}\leq \exp(\exp(c[w]_{A_\infty ^+}))$.
\end{corollary}

We close this introductory section with a few comments concerning our motivation for studying these estimates. In \cites{KS,Saari} the third author has considered higher dimensional weighted norm inequalities for one-sided maximal operators. These arise naturally in the study of solutions of appropriate doubly nonlinear partial differential equations. Other approaches to defining and studying higher dimensional one-sided maximal operators appear for example in  \cite{FMO},\cite{LerOmb}, and \cite{Ombrosi}. A common feature of all these studies is that, in the higher dimensional case, the connection between the one-sided classes $A_p ^+$ and $A_\infty ^+$ remains elusive. This should be contrasted to the one-dimensional case where we have a full analogue of the two-sided Muckenhoupt theory.

It is customary for most of the one-dimensional papers studying $A_\infty ^+$ in the literature to set everything up with respect to some other measure $g$ and thus study the classes $A_p ^+(g)$. This setup is particularly suited for some symmetry arguments that allow one to show that a weight in $A_\infty ^+$ belongs to $A_p ^+$ for some $p>1$. This approach does not seem to be available in higher dimensions as many general results that hold for arbitrary measures in one-dimension fail  in higher dimensions. The current paper avoids this setup and thus proposes another self-contained strategy for proving that $\bigcup_{p>1} A_p^+=A_\infty ^+$. We plan to pursue the higher dimensional analogues in a future work.

\section{Notation} We use the letters $C,c$ to denote generic positive constants which might change even in the same line of text. We write $A\lesssim B$ if $A\leq C B$ and $A\eqsim B$ if $A\lesssim B$ and $B\lesssim A$. Throughout the text $w$ is a nonnegative, locally integrable function on the real line and we write $w(a,b)\coloneqq \int_{(a,b)}w(t)dt$. Finally, given $\beta\in(0,1)$ and an interval $(a,b)\subset \R$ we say that a function $f:(a,b)\to \R$ lies in the local H\"older class $C^\beta(a,b)$ if for all compact $K\subset (a,b)$ and all $x,y\in K$ we have that
$|f(x)-f(y)|\lesssim_K |x-y|^\beta$ for all $x,y\in K$.

\section{Preliminaries} In this section we collect classical results about one-sided Muckenhoupt weights that we will need throughout the paper.

The first is a version of the classical rising sun lemma adjusted to our setup and is a minor modification of \cite{Saw}*{Lemma 2.1}.

\begin{lemma}\label{l.rsun} Let $\lambda>0$ and $f\geq 0$ with compact support. Then
	\[
	\{x\in\R:\, \M^+f(x)>\lambda\}=\bigcup_j (a_j,b_j)
	\]
where the intervals $\{(a_j,b_j)\}_j$ are pairwise disjoint and for every $j$ we have that $\fint _x ^{b_j} f> \lambda $ for all $x\in(a_j,b_j)$. In particular $\fint_{a_j} ^{b_j} f=\lambda$ for all $j$.
\end{lemma}

\begin{proof}Let $E_\lambda\coloneqq \{x\in\R:\, \M^+f(x)>\lambda\}$.  By \cite{Saw}, we have that  $E_\lambda=\bigcup_j (a_j,b_j)$ where $\{(a_j,b_j)\}_j$ is a disjoint collection such that $\fint_{a_j}^{b_j}f = \lambda$ for every $j$ and such that $\fint_{x}^{b_j}f \geq \lambda$ whenever $x \in (a_j, b_j)$ for some $j$. Now, for every $x\in (a_j,b_j)$ there exists $r>x$ such that $\fint_x ^r f>\lambda.$ If $r=b_j$ then we are done. If $r>b_j$ then since $\fint_{b_j} ^r f\leq \lambda$ we can conclude again that $\fint_x ^{b_j}f>\lambda$. Finally, if $r<b_j$ we consider the maximal $s\geq r$ such that $\fint_r ^s f\geq\lambda$. Necessarily $s\geq b$ due to the maximality of $s$. From this and the fact that $\fint_{b_j} ^sf\leq \lambda$ we conclude that $\fint_r ^{b_j} f\geq \lambda$. Thus $\int_x ^{b_j} f=\int_x ^r f +\int_r ^{b_j}f >\lambda(b_j-x)$ and we are done.
\end{proof}
 Whenever we have a one-sided weight $w\in A_\infty ^+$ we will need to use quantitative estimates of the type
\[
\frac{w(E)}{w(a,c)}\lesssim \big(\frac{|E|}{|(b,c)|}\big)^\delta
\]
whenever $a<b<c$ and $E\subseteq (a,b)$. In fact it is well known that the estimate above provides an equivalent definition for the class $A_\infty ^+$; see for example \cite{MPT}. Another equivalent way to define $A_\infty ^+$ goes through the validity of appropriate reverse H\"older inequalities.

The following theorem summarizes these equivalences. Note that the direct implications (i) and (ii) below are directly taken from \cite{MT}. The optimality of these estimates for $w\in A_\infty ^+$, contained in (iii), appears to be new as the authors in \cite{MT} didn't pursue this direction.
\begin{theorem}\label{t.ainfty} Let $w$ be a nonnegative, locally integrable function on $\R$. Then the following hold.
	\begin{enumerate}
\item[(i)] If $w\in A_\infty ^+$ then for $0<\epsilon\leq \frac{1}{2[w]_{A_\infty ^+}}$ and for all $a<b<c$ we have the one-sided reverse H\"older inequality
\[
|(b,c)|^\epsilon \int_{(a,b)} w^{1+\epsilon}\leq 2 \bigg(\int_{(a,c)} w \bigg)^{1+\epsilon}.
\]
\item[(ii)] If $w\in A_\infty ^+$ then for $0<\epsilon\leq \frac{1}{2[w]_{A_\infty ^+}}$ and for all $a<b<c$ and for all measurable sets $E\subseteq (a,b)$ we have
\[
\frac{w(E)}{w(a,c)}\leq 2 \big(\frac{|E|}{|(b,c)|}\big)^{\frac{\epsilon}{1+\epsilon}}\leq 2 \big(\frac{|E|}{|(b,c)|}\big)^{\frac{1}{3[w]_{A_\infty ^+}}}.
\]
\item[(iii)] Conversely, if the conclusion of (i) or (ii) holds for some $0<\epsilon<1$ then $w\in A_\infty ^+$ and $[w]_{A_\infty ^+}\lesssim \frac{1}{\epsilon}$; thus (i), (ii) are best possible up to numerical constants.
\end{enumerate}
\end{theorem}

\begin{proof} The statement and proof of (i) is \cite{MT}*{Theorem 3.4}. Part (ii) follows trivially by an application of the standard H\"older inequality.
	
Let us move to the proof of (iii). First we observe that the validity of (ii) for some $0<\epsilon<1$ implies the reverse H\"older inequality
\[
|(b,c)|^\frac{\epsilon}{2}\int_{(a,b)} w^{1+\epsilon/2} \leq 6 \bigg(\int_{(a,c)}w\bigg)^{1+\epsilon/2}
\]
To see this we can assume that $w(a,c)/|(b,c)|=1$. Setting $E_\lambda\coloneqq\{x\in(a,b):\, w(x)>\lambda\}$ we have
\[
\frac{|E_\lambda|}{|(b,c)|}\leq \frac{1}{\lambda}\frac{w(E_\lambda)}{w(a,c)}\leq 2 \frac{1}{\lambda}\frac{|E_\lambda|^\frac{\epsilon}{1+\epsilon}}{|(b,c)|^\frac{\epsilon}{1+\epsilon}}
\]
and thus we have proved the estimate
\[
\frac{|E_\lambda|}{|(b,c)|}\leq 2 ^{1+\epsilon}\frac{1}{\lambda^{1+\epsilon}}.
\]
Using (ii) again this implies $\frac{w(E_\lambda)}{w(a,c)}\leq 2^{1+\epsilon}\frac{1}{\lambda^\epsilon}$. Now 
\[\begin{split}
\int_{(a,b)} w^{1+\epsilon/2}&=\int_{(a,b)}w^\frac{\epsilon}{2}w \leq w(a,b)+\frac{\epsilon}{2}\int_1 ^\infty \lambda^{\epsilon/2-1} w(E_\lambda) d \lambda
\\
&\leq w(a,b)+2^{\epsilon} \epsilon\int_1 ^\infty\frac{w(a,c)}{\lambda^{1+\epsilon/2}}d\lambda  \leq 5 w(a,c)=5|(b,c)|.
\end{split}
\]
This shows the claimed reverse H\"older inequality.

Now we show that a reverse H\"older inequality with exponent $1+\epsilon$ implies that $[w]_{A_\infty ^+}\lesssim 1/\epsilon$. For this fix $a<b<c$ and set $r\coloneqq 1+\epsilon$. We write $(a,b)=\cup_{j=1} ^\infty I_j$ where $I_j=(x_{j},x_{j+1}]$, and $x_0\coloneqq a$, $x_1\coloneqq \frac{a+b}{2}$ and $x_j\coloneqq \frac{x_{j-1}+b}{2}$.

Assuming that a reverse H\"older inequality, as in (i), holds with exponent $r=1+\epsilon$ we can use the bound $\|\M^-\|_{L^r\to L^r}\lesssim r'$ to estimate
\[
\begin{split}
\int_{(a,b)}\M^-(w\ind_{(a,b)})&\leq   \sum_{j=1} ^\infty \int_{(a,b)} \M^-(w\ind_{I_j})=\sum_{j=1} ^\infty \int_{\cup_{k\geq j} I_k} \M^-(w\ind_{I_j})
\\
	& \lesssim r' \sum_{j=1} ^\infty \Abs{\bigcup_{k\geq j}I_k}^\frac{1}{r'}  \bigg( \int_{I_j}w^r\bigg)^\frac{1}{r}
	\\
	&\eqsim r'\sum_{j=1} ^\infty |I_j|^\frac{1}{r'} \bigg(\int_{I_j\cup I_{j+1}} w\bigg) |I_{j+1}|^{-\frac{1}{r'}}\lesssim r' w(a,b)
\end{split}
\]
and thus $[w]_{A_\infty ^+}\lesssim r'\eqsim \frac{1}{\epsilon}$ as we wanted.
\end{proof}

\section{Proof of the main theorem}
We divide the proof of the main theorem into two parts. In the first part we show Corollary~\ref{c.main1} which also shows that (i) implies (ii) in Theorem~\ref{t.main}.

\begin{proof}[Proof of Corollary~\ref{c.main1}] Let $E\subset \R$ be a compact subset of the real line with $0<w(E)<+\infty$ and let $\alpha>0$. We use Lemma~\ref{l.rsun} to write
	\[
	\{x\in\R:\, \M^+ \ind_E(x)>\alpha\}=\bigcup_j (a_j,b_j),
	\]
where $\fint_{x}^{b_j} \chi_E > \alpha$ for all $x \in (a_j, b_j)$ and $\fint_{a_j}^{b_j}\chi_E = \alpha$ for each $j$.
Let $(a,b)\in\{(a_j,b_j)\}_j$. By Lemma~\ref{l.rsun} we have that $\fint_{a} ^{b} \ind_E=\alpha$ and $\fint_x ^b \ind_E > \alpha $ for all $x\in (a,b)$. Using an idea from \cite{Reyes} we choose an increasing sequence $\{x_k\}_{k=0} ^\infty$ such that $x_0\coloneqq a$, $(a,b)=\cup_{k\geq 1}(x_{k-1},x_k]$, and $\int_{x_{k-1}} ^{x_k} \ind_E =\int_{x_k} ^b \ind_E$ for all $k\geq 1$. Then we can estimate for all $k\geq 1$
\[
\begin{split}
w((x_{k-1},x_k]\setminus E)	&=\frac{w((x_{k-1},x_k]\setminus E)}{w(x_{k-1},x_{k+1})} w(x_{k-1},x_{k+1})
\\
&\lesssim \bigg(\frac{|(x_{k-1},b)\setminus E|}{|(x_k,x_{k+1})|}\bigg)^\frac{1}{3[w]_{A_\infty ^+}} w(x_{k-1},x_{k+1})
\end{split}
\]
by (ii) of Theorem~\ref{t.ainfty}. By Lemma~\ref {l.rsun} we can conclude that $|(x_{k-1},b)\setminus E|\leq \frac{1-\alpha}{\alpha}\int_{x_{k-1}} ^b \ind_E$. Remembering the definition of the sequence $\{x_k\}_k$ we can further calculate
\[
\begin{split}
\int_{x_{k-1}} ^b\ind_E=\int_{x_{k-1}} ^{x_k} \ind_E + \int_{x_k} ^b \ind_E = 2\int_{x_k} ^b \ind_E.
\end{split}
\]
By a single recursion of this formula and another use of the definition of the sequence $\{x_k\}_k$ we thus have
\[
\int_{x_{k-1}} ^b\ind_E=4\int_{x_{k+1}} ^b \ind_E =4\int_{x_k} ^{x_{k+1}}\ind_E\leq 4(x_{k+1}-x_k).
\]
We have proved that
\[
w((x_{k-1},x_{k}]\setminus E)\lesssim (1-\alpha)^{(3[w]_{A_\infty ^+})^{-1}}w(x_{k-1},x_{k+1}).
\]
Summing over $k\geq 1$ we conclude that for every $j$ we have
\[
w((a_j,b_j)\setminus E))\lesssim (1-\alpha)^{(3[w]_{A_\infty ^+})^{-1}} w(a_j,b_j).
\]
Summing over $j$ we get the desired estimate
\[
\Cw(\alpha)-1 \lesssim (1-\alpha)^{(c[w]_{A_\infty ^+})^{-1}}
\]
whenever $\alpha>1-e^{-c[w]_{A_\infty ^+}}$, for some numerical constant $c>1$.

We now proceed to exhibit the optimality of the Solyanik estimate just proved. For this let us assume that we have an estimate $\Cw(\alpha)-1\lesssim(1-\alpha)^\frac{1}{\beta}$ for $\alpha>1-e^{-\beta}$. We will prove that for all $a<b<c$ and measurable $E\subset (a,b)$ we have the estimate
\[
\frac{w(E)}{w(a,c)}\lesssim \bigg(\frac{|E|}{|(b,c)|}\bigg)^\frac{1}{\beta}
\]
By Theorem~\ref{t.ainfty} this will imply that $w\in A_\infty ^+$ and $[w]_{A_\infty ^+}\lesssim \beta $ as claimed.

We now fix real numbers $a<b<c$ and consider a measurable set $E\subseteq (a,b)$. We set $E'\coloneqq (a,c)\setminus E$ and consider two cases.

If $|E|/|(b,c)|\geq e^{-\beta}$ then
\[
\frac{w(E)}{w(a,c)}\leq 1 = e e^{-1}\leq e \bigg(\frac{|E|}{|(b,c)|}\bigg)^\frac{1}{\beta}
\]
and we are done.

In the complementary case we have for all $x\in(a,b)$
\[
\begin{split}
\M^+\ind_{E'}(x)&\geq \frac{|E'\cap(x,c)|}{|(x,c)|}=\frac{|(b,c)|+|(x,b)|-|(x,b)\cap E|}{|(x,c)|}\geq 1 -\frac{|(x,b)\cap E|}{|(x,c)|}
\\
& \geq 1 -\frac{|E|}{|(b,c)|}>1-e^{-\beta}.
\end{split}
\]
Obviously $\M^+\ind_{E'}=1$ on $(b,c)$ and thus $(a,c) \subseteq \big\{x\in\R:\, \M^+\ind_{E'}(x)\geq 1-|E|/|{b,c}| \big\}$. As $1-|E|/|(b,c)|>1-e^{-\beta}$ we can use the assumed Solyanik estimate to conclude that
\[
w(a,c)\leq \Cw\big(1-|E|/|(b,c)|\big)w(E')=\Cw\big(1-|E|/|(b,c)|\big) (w(a,c)-w(E))
\]
and thus
\[
w(E)\leq \big(\Cw\big(1-|E|/|(b,c)|\big)-1\big)w(a,c)\lesssim \big(\frac{|E|}{|(b,c)|}\big)^\frac{1}{\beta}w(a,c)
\]
as we wanted.
\end{proof}

It is trivial that (ii) implies (iii) in Theorem~\ref{t.main}  so we move on to prove that (iii) implies (i). It will clearly suffice to show the following.

\begin{proposition}\label{p.restricted} Suppose that $w$ is a non-negative, locally integrable function that satisfies $\Cw(\alpha_o)<+\infty$ for some $\alpha_o\in(0,1)$. Then $\M^+$ is of restricted weak type $(p,p)$ with respect to $w$ for $p=\log\Cw(\alpha_o)/\log(1/\alpha_o)$, with constant $\Cw(\alpha_o)^\frac{1}{p}$.
\end{proposition}
The proof of Proposition~\ref{p.restricted} relies on the notion of the \emph{Halo extension} of a set $E$, defined as follows. Given $\lambda\in(0,1)$ the Halo extension of $E$ is
\[
\mathcal H_\lambda ^+(E)\coloneqq \{x\in\R:\, \M^+\ind_E(x)>\lambda\}.
\]
We also set $\mathcal H^{+,0} _\lambda(E)\coloneqq E$ and for a positive integer $k>1$
\[
\mathcal H^{+,k} _\lambda(E)\coloneqq \mathcal H_\lambda ^+ (\mathcal H_\lambda ^{+,k-1}(E)).
\]
The heart of the matter is the following lemma.
\begin{lemma}\label{l.haloiter} Let $0<\lambda< \alpha<1$ and $E\subset \R$ be a measurable set with $0<w(E)<+\infty$. Then
	\[
	\mathcal H ^+ _\lambda(E)\subseteq \mathcal H^{+,N} _\alpha (E),
	\]
where
\[
N =\Big\lceil\frac{\log\frac{1}{\lambda}}{\log\frac{1}{\alpha}}\Big\rceil.
\]	
Here $\lceil x\rceil$ denotes the smallest integer which is no less than $x$.
\end{lemma}

\begin{proof} Let $(a,b)$ be one of the component intervals of $\mathcal H_\lambda ^+ (E)$ provided by Lemma~\ref{l.rsun}. The same lemma allows us to write $\mathcal H^+ _\alpha(E)$ as the union of disjoint intervals $\cup_{j} (\xi_j,\eta_j)$, where $\fint_{\xi_j}^{\eta_j} \chi_E = \alpha$. Since $\lambda<\alpha<1$ we have that $a,b\notin(\xi_j,\eta_j)$ for any $j$. Indeed, if say $a\in(\xi_j,\eta_j)$ then there exists some $h>0$ such that
	\[
	\alpha h <\int_a ^{a+h} \ind_E \leq \lambda h
	\]
since $a\notin\mathcal H^+ _\lambda(E)$. This however contradicts the condition $\lambda<\alpha$. Similarly we see that $b\notin(\xi_j,\eta_j)$. Thus $\mathcal H_{\alpha} ^+(E)\cap (a,b)=\cup_{j\in J}(\xi_j,\eta_j)$. For $x\in (a,b)\setminus \mathcal H_\alpha ^+ (E)$ let $J_x\coloneqq \{j\in J:\, \xi_j>x\}$. Then if $x\in (a,b)\setminus \mathcal  H_\alpha ^+(E) $ we can calculate
\[
|\mathcal H^+ _\alpha(E)\cap (x,b)|=\sum_{j\in J_x} |(\xi_j,\eta_j)|=\frac{1}{\alpha}\Abs{\bigcup_{j\in J_x}(\xi_j,\eta_j)\cap E}=\frac{1}{\alpha}|E\cap (x,b)|>\frac{\lambda}{\alpha}|(x,b)|,
\]
the last inequality following by Lemma~\ref{l.rsun}. Thus $(a,b)\subseteq \mathcal H^+ _{\frac{\lambda}{\alpha}}(\mathcal H^+ _\alpha(E))$ and accordingly $\mathcal H^+ _\lambda(E)\subseteq \mathcal H^+ _{\frac{\lambda}{\alpha}}(\mathcal H^+ _\alpha(E))$.

Let $K\geq 1$ be the smallest integer such that $\lambda/\alpha^K>\alpha$. Iterating the estimate above we conclude that
\[
\mathcal H_\lambda ^+ (E)\subseteq \mathcal H_{\frac{\lambda}{\alpha^K}} ^+ (\mathcal H_\alpha ^{+,K} (E))\subseteq \mathcal H_\alpha ^{+,K+1}(E).
\]
Note that $K$ satisfies $\lambda/\alpha^{K-1}< \alpha$ and thus $K+1 = \lceil \log(1/\lambda)/\log(1/\alpha) \rceil.$
\end{proof}

We can now give the proof of Proposition~\ref{p.restricted}.

\begin{proof}[Proof of Proposition~\ref{p.restricted}] Let $\lambda\in(0,1)$. Since $\Cw(\alpha_o)<+\infty$ we trivially get that for all $\lambda>\alpha_o$ and every $E\subset \R$
	\[
	w(\{x\in\R:\,\M^+\ind_E(x)>\lambda\})\leq \Cw(\alpha_o) w(E)\leq \frac{\Cw(\alpha_o)}{\lambda^q}w(E)
	\]
for any $q\geq 1$. It thus suffices to consider the case $0<\lambda<\alpha_o<1$. Let $1>\alpha>\alpha_o$. Then for any set $S\subset \R$ we have
\[
w(\mathcal H_\alpha ^+(S))\leq \Cw(\alpha_o)w(S).
\]
By Lemma~\ref{l.haloiter} applied for $\alpha>\lambda>0$ we can conclude that
\[
w(\mathcal H_\lambda ^+(E))\leq \Cw(\alpha_o) ^N w(E)
\]
with $N = \lceil {\log\frac{1}{\lambda}}/{\log\frac{1}{\alpha}} \rceil \leq {\log\frac{1}{\lambda}}/{\log\frac{1}{\alpha}} +1$. We get
\[
w(\{x\in\R:\,\M^+\ind_E(x)>\lambda\})\leq \Cw(\alpha_o) \big(\frac{1}{\lambda}\big)^{\frac{\log\Cw(\alpha_o)}{\log\frac{1}{\alpha}}}w(E)
\]
for all $\alpha>\alpha_o$. Letting $\alpha\to \alpha_o ^+$ we conclude that $\M^+$ is of restricted weak type $(p,p)$ with respect to $w$ for $p=\log\Cw(\alpha_o)/\log(1/\alpha_o)$, with constant $\Cw(\alpha_o)^\frac{1}{p}$.
\end{proof}

We conclude this section with the proof of Corollary~\ref{c.main2}

\begin{proof}[Proof of Corollary~\ref{c.main2}] Assuming that $w\in A_\infty ^+$ with constant $[w]_{A_\infty ^+}$, Theorem~\ref{t.main} implies that $\Cw(\alpha_o)\lesssim 1$ for $\alpha_o=1-e^{-c[w]_{A_\infty ^+}}$, and $c>1$ is a numerical constant. Now Proposition~\ref{p.restricted} implies that $\M^+$ is of restricted weak type $(p,p)$ with respect to $w$ for $p \eqsim (\log(1/\alpha_o))^{-1}\eqsim e^{c[w]_{A_\infty ^+}}$ and thus $w\in A_p ^+$ for $p>e^{c[w]_{A_\infty ^+}}$. By carefully examining the interpolation constants and using the estimate
	\[
	\|\M^+ \|_{L^p(w)\to L^p(w)}\gtrsim [w]_{A_p ^+} ^\frac{1}{p}
	\]
which is contained in \cite{MT}*{Theorem 1.4} we can conclude that $[w]_{A_p ^+}\leq e^{e^{c[w]_{A_\infty ^+}}}$ for some numerical constant $c>0$. See also \cite{HP2}*{p. 21} for the details of this calculation.
\end{proof}

\section{Local H\"older continuity} The methods of  this paper easily imply that $\Cw$ is locally H\"older continuous on $(0,1)$ whenever $w\in A_\infty ^+$.

\begin{corollary} There exists a numerical constant $c>1$ such that for all $w\in A_\infty ^+$ we have $\Cw \in C^\frac{1}{ c[w]_{A_\infty ^+} }(0,1)$.
\end{corollary}

\begin{proof} Let $E\subset\R$ and $0<\lambda<\alpha<1$. By the proof of Lemma~\ref{l.haloiter} we have that
	\[
	\mathcal H^+ _\lambda(E)\subseteq \mathcal H^+ _\frac{\lambda}{\alpha}(\mathcal H^+ _\alpha(E)).
	\]
We conclude that
\[
0<\Cw(\lambda)-\Cw(\lambda/\alpha)\leq\Cw(\lambda/\alpha)(\Cw(\alpha)-1).
\]	
Using Corollary~\ref{c.main1} we get that for all $0<x<y<1$ we have
\[
|\Cw(x)-\Cw(y)|\lesssim \Cw(y)\big(\frac{y-x}{x}\big)^\frac{1}{c[w]_{A_\infty ^+}}
\]
for some numerical constant $c>0$. The local H\"older continuity of $\Cw$ follows immediately from the estimate above.
\end{proof}

\section*{Acknowledgment} We are indebted to Francisco Javier Mart{\'{\i}}n-Reyes for enlightening discussions related to the subject of the paper.


\end{document}